\documentclass[11pt,a4paper]{article}
\usepackage{amssymb,amsmath}
\usepackage{graphicx}

\newtheorem{theorem}{Theorem}[section]
\newtheorem{corollary}[theorem]{Corollary}

\newtheorem{lemma}[theorem]{Lemma}
\newtheorem{proposition}[theorem]{Proposition}
\newtheorem{remark}[theorem]{Remark}

\newtheorem{example}[theorem]{Example}

\newenvironment{proof}{\begin{trivlist}\item[]{\it
Proof.}}{\hfill$\square$\end{trivlist}}

\def\field{k}
\def\mn{{\mathbb{N}}}
\def\gl{{\mathrm{Gl}}}
\def\rep{{\mathcal{R}}}
\def\mz{{\mathbb{Z}}}
\def\sstheta{{\theta-{\mathrm{ss}}}}
\def\stheta{{\theta-{\mathrm{s}}}}
\def\ssvntheta{{\theta^{v,n}-{\mathrm{ss}}}}
\def\moduli{{\mathcal{M}}}
\def\stab{{\mathrm{Stab}}}

\begin{document}
\title{On singularities of quiver moduli} 
\author{M\'aty\'as Domokos
\thanks{Partially supported by 
OTKA NK72523 and K61116. } 
\\ 
\\
{\small R\'enyi Institute of Mathematics, Hungarian Academy of 
Sciences,} 
\\ {\small P.O. Box 127, 1364 Budapest, Hungary,} {\small E-mail: domokos@renyi.hu } 
}
\date{}
\maketitle 
\begin{abstract}  
Any moduli space of representations of a quiver (possibly with oriented cycles) has an embedding as a dense open subvariety into a moduli space of representations of a bipartite quiver having the same type of singularities. A connected quiver is Dynkin or extended Dynkin 
if and only if all moduli spaces of its representations are smooth. 
\end{abstract}



\section{Introduction}

A quiver $Q$ is a finite directed graph with vertex set $Q_0$ and arrow set $Q_1$. For an arrow $a\in Q_1$ write $a_-\in Q_0$ for its starting vertex, and 
$a_+$ for its terminating vertex (multiple arrows, oriented cycles, loops are 
allowed). Let $\field$ be an algebraically closed field of arbitrary characteristic. Take a dimension vector $\alpha\in\mn_0^{Q_0}$ (here $\mn_0$ stands for the set of non-negative integers). The space of $\alpha$-dimensional representations of $Q$ is defined as 
$\rep(Q,\alpha):=\bigoplus_{a\in Q_1}\field^{\alpha(a_+)\times \alpha(a_-)}$, 
so $x\in \rep(Q,\alpha)$ assigns an $\alpha(a_+)\times \alpha(a_-)$ matrix $x(a)$ to each arrow $a\in Q_1$. 
For an element $g=(g(i)\mid i\in Q_0)$ in the product $\gl(\alpha):=\prod_{i\in Q_0}\gl_{\alpha(i)}(\field)$ of general linear groups  
and $x\in \rep(Q,\alpha)$ define $g\cdot x \in \rep(Q,\alpha)$ 
by the rule $(g\cdot x)(a):=g(a_+)x(a)g(a_-)^{-1}$ (matrix multiplication). 
This is a linear action of $\gl(\alpha)$ on $\rep(Q,\alpha)$, such that the orbits are in a natural bijection with the isomorphism classes of $\alpha$-dimensional representations of $Q$ (see for example \cite{kraft-riedtmann} for the concept of the category of representations of $Q$). 
By a {\it weight} we mean an integral vector $\theta\in\mz^{Q_0}$; 
a {relative invariant of weight} $\theta$ is a polynomial function 
$f$ on $\rep(Q,\alpha)$ satisfying the property 
$f(g\cdot x)=\prod_{i\in Q_0}\det(g(i))^{\theta(i)}f(x)$ 
for all $g\in \gl(\alpha)$ and $x\in \rep(Q,\alpha)$. 
A point $x\in\rep(Q,\alpha)$ is $\theta$-semistable if there exists a 
relative invariant $f$ whose weight is a positive rational multiple of $\theta$ and $f(x)$ is non-zero. The $\theta$-semistable points constitute a  Zariski open (possibly empty) subset $\rep(Q,\alpha)^\sstheta$ in $\rep(Q,\alpha)$. 
A $\theta$-semistable point is $\theta$-stable if its stabilizer is $\field^{\times}$, they consitute an open subset $\rep(Q,\alpha)^\stheta$ in $\rep(Q,\alpha)^\sstheta$. 
In \cite{king}, Geometric Invariant Theory is applied to construct a morphism 
$\pi(Q,\alpha,\theta):\rep(Q,\alpha)^\sstheta\to \moduli(Q,\alpha,\theta)$ onto a quasiprojective algebraic variety $\moduli(Q,\alpha,\theta)$, which is a coarse moduli space for $\alpha$-dimensional $\theta$-semistable representations of $Q$ up to S-equivalence (consult \cite{king}, \cite{newstead} for the terminology). Moreover, $\moduli(Q,\alpha,\theta)$ contains a (possibly empty)  open subset $\moduli^s(Q,\alpha,\theta)$ which is a coarse moduli space for $\alpha$-dimensional $\theta$-stable representations up to isomorphism. 
Note also that the notion of semistability (stability) and the 
associated moduli spaces depend only on the equivalence class of $\theta$, 
where two weights are said to be equivalent if one is a positive rational multiple of the other. 

It is known that the moduli spaces $\moduli(Q,\alpha,\theta)$ are singular in general, see for example the introduction of \cite{engel-reineke}, or the  analysis of the generalized Kronecker quiver in \cite{adriaenssens-lebruyn}. 
One possible way to give this vague statement a concrete form is provided by our Theorem~\ref{thm:main}, pointing out that Dynkin or extended Dynkin quivers are characterized by the property that all their moduli spaces are smooth (in fact they are all affine or projective spaces). 

In the representation theory of quivers, the classical distinction of the classes of Dynkin (resp. extended Dynkin) quivers is based on the fact that they have finite (resp. tame) representation type, whereas all the remaining quivers have wild representation type. More recent works showed that exactly these classes are selected when one inquires about good algebraic or combinatorial 
properties of associated objects. It is shown in \cite{skowronski-weyman} 
that the (extended) Dynkin quivers are exactly those quivers that have the property that all their algebras of semi-invariants are complete intersections.  These quivers are characterized in \cite{chindris} in terms of their weight semi-groups. It is quite natural to inquire about a characterization of extended Dynkin quivers by good geometric properties of their moduli spaces; 
Theorem~\ref{thm:main} provides the simple answer. 

In much of the literature on moduli spaces of quivers the authors 
require that the quiver has no oriented cycles. We show in Section~\ref{sec:doubling} that a moduli space attached to an arbitrary quiver can be embedded as a dense open subvariety into a moduli space of a bipartite quiver, such that this larger moduli space has the same type of singularities as the original one. Thus studying certain questions on these moduli spaces, one may reduce to the case when $Q$ has no oriented cycles (or even to the case when $Q$ is bipartite). 
Recall that if the quiver $Q$ has no oriented cycles, then $\moduli(Q,\alpha,\theta)$ is a projective variety. This is not true for quivers containing oriented cycles.
So one may think of this process as a compactification of the original moduli space, and it is notable that such compactification is possible without adding new type of singularities. 

Sections~\ref{sec:doubling} and \ref{sec:smooth} are essentially independent 
(though the idea of Theorem~\ref{thm:embedding} is used to allow quivers with oriented cycles in the statements of Section~\ref{sec:smooth}). 


\section{The effect on moduli spaces of doubling a vertex}  \label{sec:doubling}

Pick a vertex $v\in Q_0$ and construct a new quiver $Q^v$ as follows: 
replace the vertex $v$ of $Q$ by two new vertices $v_-$ and $v_+$, and keep all the other vertices. For each arrow $a\in Q_1$ draw an arrow $a^v\in Q^v_1$ 
with the same endpoints as $a$, except that $a^v_-=v_-$ when $a_-=v$, 
and $a^v_+=v_+$ when $a_+=v$ (in particular, if $a$ is a loop at $v$, then 
$a^v$ is an arrow from $v_-$ to $v_+$). Moreover, $Q_1^v$ has an extra arrow $e$ from $v_-$ to $v_+$. 
If $\alpha\in\mn_0^{Q_0}$ is a dimension vector, 
then denote $\alpha^v$ the dimension vector for $Q^v$ with 
$\alpha^v(i)=\alpha(i)$ for all $i\in Q_0\setminus \{v\}$, 
and $\alpha^v(v_-)=\alpha(v)=\alpha^v(v_+)$. 
For a weight $\theta\in\mz^{Q_0}$ and a non-negative integer $n$, denote by 
$\theta^{v,n}$ the weight for $Q^v$ defined by 
$\theta^{v,n}(v_-)=-n$, $\theta^{v,n}(v_+)=\theta(v)+n$, and 
$\theta^{v,n}(i)=\theta(i)$ for all $i\in Q_0\setminus \{v\}$.  
Let $\iota:\rep(Q,\alpha)\to\rep(Q^v,\alpha^v)$ be the morphism with 
$\iota(x)(a^v)=x(a)$ for $a\in Q_1$, and $\iota(x)(e)=I$ (the $\alpha(v)\times\alpha(v)$ identity matrix). 
We state first a variant (taking care of weights) of Theorem 3.2 in \cite{domokos-zubkov2001} or Proposition 1 in \cite{derksen-weyman2002} 
(see also \cite{domokos} for a special case).

\begin{proposition}\label{prop:lifting} 
Let $f$ be a relative invariant on $\rep(Q,\alpha)$ with weight $\theta$, and assume that $f$ is homogeneous of total degree $d$ in the entries belonging to  $\{x(a)\mid a_-=v\}$. Then there is a relative invariant $\tilde f$ on $\rep(Q^v,\alpha^v)$ with weight $\theta^{v,d}$ such that $f=\tilde{f}\circ\iota$. 
\end{proposition} 

\begin{proof} Denote $M^*$ the adjugate of an $l\times l$ matrix $M$: 
the $(i,j)$-entry of $M^*$ is $(-1)^{i+j}$-times the determinant of the $(l-1)\times(l-1)$ minor of $M$ obtained by omiting the $j$th row and the $i$th column. When $M$ is invertible, then $M^*=\det(M)M^{-1}$. 
This shows that $(AMB^{-1})^*=\det(B)^{-1}\det(A)BM^*A^{-1}$ for $A,B\in\gl_l(\field)$. 
Define the morphism $\Phi:\rep(Q^v,\alpha^v)\to\rep(Q,\alpha)$ by 
$$\Phi(x)(a)=
\begin{cases}x(a^v) \mbox{ when } a_-\neq v \\
\Phi(x)(a)=x(a^v)\cdot x(e)^* \mbox{ when } a_-=v.\end{cases}$$ 
Given $g\in\gl(\alpha^v)$ define 
$\bar g\in\gl(\alpha)$ by $\bar g(i)=g(i)$ for $i\neq v$ and 
$\bar g(v)=g(v_+)$. For $x\in\rep(Q^v,\alpha^v)$ one has the formula 
\begin{equation}\notag
\Phi(g\cdot x)(a)=
\begin{cases}(\bar g\cdot\Phi(x))(a)\quad\mbox{ when }a_-\neq v \\
\det(g(v_+))\det(g(v_-))^{-1} (\bar g\cdot \Phi(x))(a)\quad \mbox{ when }a_-=v.
\end{cases}  
\end{equation}
This shows that $\tilde{f}:=f\circ \Phi$ is a relative invariant on $\rep(Q^v,\alpha^v)$ with weight $\theta^{v,d}$.  It has the property that $\tilde{f}(\iota(x))=f(x)$ for all $x\in\rep(Q,\alpha)$. 
\end{proof} 

Next we recall the concept of the {\it local quiver setting of} 
$\xi\in\moduli(Q,\alpha,\theta)$ from \cite{adriaenssens-lebruyn}.  
The fibre $\pi^{-1}(\xi)$ contains a unique closed orbit (closed in $\rep(Q,\alpha)^\sstheta$), say the orbit of $x$. 
Then the representation $V_x$ of $Q$ corresponding to $x$ decomposes as 
$\bigoplus_{i=1}^qm_iV_i$, where $V_1,\ldots,V_q$ are pairwise non-isomorphic $\theta$-stable representations, and $m_i\in\mn$ stands for the multiplicity 
of $V_i$ as a summand. Denote by $\beta_i$ the dimension vector of $V_i$. 
Then $\tau:=(\beta_1,m_1;\ldots;\beta_q,m_q)$ is called the {\it $\theta$-semistable 
representation type} of $\xi$ (note that it may happen that $\beta_i=\beta_j$ 
for some $i\neq j$, when there are non-isomorphic $\theta$-stable representations of dimension vector $\beta_i=\beta_j$). 
The {\it local quiver setting} associated to $\xi$ depends on the representation type $\tau$ of $\xi$, and it consists of a quiver 
$Q_{\xi}$ with vertex set $\{1,\ldots,q\}$, together with the dimension vector $\mu_{\xi}:=(m_1,\ldots,m_q)$. The quiver $Q_{\xi}$ has  
$\delta^i_j-\langle\beta_i,\beta_j\rangle_Q$ arrows from $i$ to $j$, where 
$\delta^i_j=0$, if $i\neq j$, $\delta^i_i=1$, and 
$\langle \alpha,\beta\rangle_Q=\sum_{i\in Q_0}\alpha(i)\beta(i)
-\sum_{a\in Q_1}\alpha(a_-)\beta(a_+)$ is the Ringel bilinear form on $\mz^{Q_0}$. 
By the {\it local quiver settings of} $\moduli(Q,\alpha,\theta)$ we mean the 
(finite) set of local quiver settings $(Q_{\xi},\mu_{\xi})$ that occur as the local quiver setting associated to some point $\xi\in \moduli(Q,\alpha,\theta)$. 

\begin{theorem} \label{thm:embedding} For a sufficiently large non-negative integer $n$, the map $\iota$ induces an isomorphism from the moduli space $\moduli(Q,\alpha,\theta)$ onto a Zariski open dense subvariety of $\moduli(Q^v,\alpha^v,\theta^{v,n})$. 
This isomorphism maps $\moduli^s(Q,\alpha,\theta)$ onto a dense open subset of   
$\moduli^s(Q^v,\alpha^v,\theta^{v,n})$. 
Moreover, for sufficiently large $n$, the local quiver settings of $\moduli(Q^v,\alpha^v,\theta^{v,n})$ and $\moduli(Q,\alpha,\theta)$ coincide.  
\end{theorem}

The proof will be divided into three Lemmas.  In order to simplify notation, set 
$\rep:=\rep(Q,\alpha)$, $\rep^{ss}:=\rep(Q,\alpha)^\sstheta$, $\rep^{s}:=\rep(Q,\alpha)^\stheta$, 
$\moduli:=\moduli(Q,\alpha,\theta)$, $\moduli^s:=\moduli^s(Q,\alpha,\theta)$, 
 $\pi:=\pi(Q,\alpha,\theta)$, $G:=\gl(\alpha)$, 
and denote $\tilde{\rep}$,  $\tilde{\rep}^{ss}$,  $\tilde\rep^{s}$, 
$\tilde{\moduli}$, $\tilde\moduli^s$,  
$\tilde\pi$, $\tilde  G$ the corresponding objects for $Q^v$, 
$\alpha^v$, and $\theta^{v,n}$. 

We say that a dimension vector $\beta$ is $\theta$-semistable (resp. $\theta$-stable), if $\moduli(Q,\beta,\theta)$ (resp. $\moduli^s(Q,\beta,\theta)$) is non-empty. By Proposition 6.7 in \cite{derksen-weyman}, $\alpha$ is 
$\theta$-semistable (stable) if and only if $\alpha^v$ is $\theta^{v,n}$-semistable (stable) for sufficiently large $n$. 
First we need to strengthen this statement as follows: 

\begin{lemma}\label{lemma:iotaim} (i) For sufficiently large $n$ we have 
$\iota(\rep^{ss})=\tilde{\rep}^{ss}\cap\iota(\rep)$. 

(ii) If the conclusion of (i) holds for $n$, then 
$\iota(\rep^{s})=\tilde{\rep}^{s}\cap\iota(\rep)$.
\end{lemma}

\begin{proof}  (i) This could be proved by modifying the proof of 
Proposition 6.7 in \cite{derksen-weyman}. We give a different proof based on Proposition~\ref{prop:lifting}, yielding a bound of different nature for the necessary $n$. 
Introduce a grading on the coordinate ring of $\rep$ by specifying the 
degree of an entry of $x(a)$ to be $1$ when $a_-=v$, and $0$ when $a_-\neq v$. 
Assume that $f(x)\neq 0$ for 
some homogeneous relative invariant $f$ of weight $\sigma:=m\theta$ ($m\in\mn$). 
By Proposition~\ref{prop:lifting}, $\tilde{f}(\iota(x))\neq 0$, hence $\iota(x)$ is $\sigma^{v,d}$-semistable, where $d$ is the degree of $f$. Moreover, multiplying $f$ by the $r$th power of the relative invariant 
$y\mapsto \det(y(e))$, we obtain a relative invariant with weight $\sigma^{v,d+r}$ not vanishing on $\iota(x)$. 
This shows that $\iota(x)$ is $\sigma^{v,n}$-semistable for all $n\geq d$. 
Since $\sigma=m\theta$, this clearly implies that $\iota(x)$ is $\theta^{v,n}$-semistable for all $n\geq d/m$. 

Now take a finite set $f_1,\ldots,f_q$ of relative invariants with weight equivalent to $\theta$, whose common zero locus in $\rep$ is the complement of  
$\rep^{ss}$. We may assume that $\theta$ is indivisible, so the weight of $f_i$ is $m_i\theta$, where $m_i\in\mn$. Since the action of $G$ preserves the grading introduced at the beginning of the proof,  the homogeneous components of a relative invariant are also relative invariants of the same weight, so we may assume  
that all the $f_i$ are homogeneous; write $d_i$ for the degree of $f_i$.  
Fix a natural number $n$  with $n\geq d_i/m_i$ for all $i=1,\ldots,q$.  
If $x\in\rep$ is $\theta$-semistable, 
then $f_i(x)\neq 0$ for some $i\in\{1,\ldots,q\}$, hence $\iota(x)$ is 
$\theta^{v,n}$-semistable by the considerations above. 

Conversely, if $\iota(x)$ is $\theta^{v,n}$-semistable, 
then $f(\iota(x))\neq 0$ for some relative $\tilde G$-invariant $f$ with weight equivalent to $\theta^{v,n}$. Identify $G$ with the subgroup $H:=\{g\in\tilde G \mid \quad g(v_-)=g(v_+)\}$ in the obvious way, and view 
$\tilde{\rep}$ as an $H\cong G$-variety. 
Then $\iota$ is $G$-equivariant, showing that $f\circ\iota$ is a relative invariant on $\rep$ with weight equivalent to $\theta$, and 
$f\circ\iota$ does not vanish on $x$, hence $x$ is $\theta$-semistable. 

(ii) If $x\in\rep^{s}$, then we know already that $\iota(x)$ is $\theta^{v,n}$-semistable, so to conclude $\iota(x)\in\tilde\rep^{s}$ it is sufficient to show that the stabilizer of $\iota(x)$ in $\tilde G$ is just the center $\field^{\times}$.  
If $g\in\tilde G$ stabilizes $\iota(x)$, then 
$g(v_+)\iota(x)(e)g(v_-)=\iota(x)(e)=I$, hence $g(v_+)=g(v_-)$. 
So $g$ belongs to the subgroup $H\cong G$ of $\tilde G$ mentioned above. Since $\iota$ is $G$-equivariant, we have 
$\stab_{\tilde G}(\iota(x))=\stab_H(\iota(x))\cong\stab_G(x)=\field^{\times}$, 
as we claimed. Conversely, if $y\in\tilde\rep^{s}\cap \iota(\rep)$, then 
$y=\iota(x)$ for some $x\in\rep^{ss}$ by (i), and the above calculation of stabilizers shows that $x\in\rep^{s}$. 
\end{proof} 

\begin{lemma}\label{lemma:imF} (i) If the conclusion of Lemma~\ref{lemma:iotaim} (i) holds for $n$, then 
there exists a unique morphism 
\quad $F:\moduli \to\tilde{\moduli} 
\quad \mbox{ with }\quad  
F\circ \pi=\tilde{\pi}\circ \iota.$  

(ii) For sufficiently large $n$, the image of $F$ is a dense open subvariety of $\tilde{\moduli}$. 

(iii) If the conclusion of (ii) holds for $n$, then $F$ gives an isomorphism between $\moduli$ and the dense open subvariety $F(\moduli)$ in $\tilde{\moduli}$. 

(iv)  If the conclusion of (ii) holds for $n$, then $F(\moduli^s)$ is a dense open subset of $\tilde\moduli^s$. 
\end{lemma} 

 \begin{proof} (i) Consider the morphism 
$\tilde{\pi} \circ \iota:\rep^{ss}
\to \tilde{\moduli}$.  
It is $G$-invariant, hence by the universal property of the quotient morphism $\pi$ (see e.g. Theorem 3.21 (i) and Proposition 3.11 (i) in \cite{newstead}), there exists a unique morphism  
$F:\moduli \to\tilde{\moduli}$  with $F\circ \pi=\tilde{\pi}\circ \iota$. 

(ii) Clearly, $U:=\tilde G\cdot \iota(\rep)$ is the dense 
$\tilde G$-stable 
affine open subset in $\tilde{\rep}$ consisting of the points $x\in\tilde{R}$ with $\det(x(e))\neq 0$. Write $U^{ss}:=U\cap \tilde\rep^{ss}$.  
Then $U^{ss}$ is dense in $\tilde\rep^{ss}$, hence   
$\tilde{\pi}(U^{ss})$ is a dense subset of $\tilde\pi(\tilde{\rep}^{ss})=\tilde{\moduli}$. 
On the other hand, 
\begin{equation}\label{eq:imF}
\tilde{\pi}(U^{ss})=\tilde\pi(\tilde G\cdot\iota(\rep)\cap\tilde{\rep}^{ss})=\tilde{\pi}(\iota(\rep)\cap \tilde{\rep}^{ss})
=\tilde{\pi}(\iota(\rep^{ss}))=F(\moduli) 
\end{equation} 
showing that $F(\moduli)$ is dense in $\tilde{\moduli}$. 

Now choose $n$ large enough so that 
$\iota(\rep(Q,\beta)^\sstheta)=\rep(Q,\beta^v)^\ssvntheta\cap\iota(\rep(Q,\beta))$ 
holds for all dimension vectors $\beta\leq\alpha$, where we write $\beta\leq\alpha$, if $\beta(i)\leq\alpha(i)$ for all $i\in Q_0$. We shall show that 
$U^{ss}$ is $\tilde\pi$-saturated, that is, $U^{ss}=\tilde\pi^{-1}(\tilde\pi(U^{ss}))$. 
Suppose that $y\in \tilde\rep^{ss}$ with $\tilde\pi(y)\in\tilde\pi(U^{ss})$. 
Then there is an $x\in\rep^{ss}$ with $\tilde\pi(y)=\tilde\pi(\iota(x))$ by 
(\ref{eq:imF}). It follows by \cite{king} that the S-equivalence class of 
$V_y$ coincides with the S-equivalence class of $V_{\iota(x)}$, where we denote by $V_z$ the representation of the quiver $Q^v$ belonging to $z\in\tilde\rep$. That is, $V_y$ and $V_{\iota(x)}$ have the same $\theta^{v,n}$-stable composition factors (i.e. Jordan-H\"older factors in the category of $\theta^{v,n}$-semistable representations of $Q^v$).  
By the choice of $n$ and by Lemma~\ref{lemma:iotaim} (ii), the $\theta^{v,n}$-stable composition factors of $V_{\iota(x)}$ (and hence of $V_y$) are $V_{\iota(x_1)},\ldots,V_{\iota(x_q)}$, 
where $V_{x_1},\ldots,V_{x_q}$ are the $\theta$-stable composition factors 
of $V_x$. It follows that replacing $y$ by an appropriate element in its $\tilde G$-orbit, we have that $y(e)$ is an upper triangular matrix with all diagonal entries equal to $1$. Consequently, $\det(y(e))\neq 0$, so $y\in U^{ss}$. 

Thus $U^{ss}$ is $\tilde\pi$-saturated. It is also open and $\tilde G$-stable. Hence $\tilde\pi(U^{ss})$ is 
the complement in $\tilde\moduli$ of the image under $\tilde\pi$ of a closed $\tilde G$-stable subset of $\tilde\rep^{ss}$. Consequently, $\tilde\pi(U^{ss})=F(\moduli)$ is open in $\tilde\moduli$ 
(see the definition of a good quotient in Chapter 3 of \cite{newstead}).  

(iii) Consider the morphism $\Psi:U\to\rep$ defined by 
$$\Psi(x)(a)=\begin{cases} x(a^v)\cdot x(e)^{-1} \mbox{ for }a\in Q_1 
\mbox{ with }a_-=v \\   
x(a^v) \mbox{ for }a\in Q_1 \mbox{ with }a_-\neq v. \end{cases}$$
Since $\Psi\circ\iota$ is the identity morphism of $\rep$, and $\Psi$ maps any  $\tilde G$-orbit into a $G$-orbit, we conclude from 
Lemma~\ref{lemma:iotaim} (i) that $\Psi(U^{ss})=\rep^{ss}$, moreover, the 
morphism $\pi\circ\Psi: U^{ss}\to \moduli$ is $\tilde G$-invariant. 
Since $U^{ss}$ is open and $\tilde\pi$-saturated, 
the map 
$\tilde{\pi}\vert_{U^{ss}}:U^{ss}\to F(\moduli)$ is a good 
$\tilde G$-quotient by Proposition 3.10 (a) in \cite{newstead}, hence is a categorical quotient by Proposition 3.11 (i) in loc. cit. 
This guarantees the existence of a unique morphism $G:\tilde{\pi}(U^{ss})\to\moduli$ with 
$G\circ \tilde{\pi}\vert_{U^{ss}}=\pi\circ \Psi\vert_{U^{ss}}$. 
Moreover, since $\Psi\circ\iota$ is the identity morphism of $\rep$, 
we get that $G\circ F$ is the identity morphism of $\moduli$. 
Consequently, $F$ is an isomorphism between $\moduli$ and the dense open 
subvariety $F(\moduli)$ in $\tilde{\moduli}$. 

(iv) We know from Lemma~\ref{lemma:iotaim} (ii) that $F(\moduli^s)=\tilde\moduli^s\cap F(\moduli)$, so being the intersection of two open sets, $F(\moduli^s)$ is open in $\tilde\moduli$. To see the density, 
it remains to show that if $\tilde\moduli^s$ is non-empty, then $\moduli^s$ is non-empty. Since $\tilde\moduli^s$ is open, if it is non-empty, then it intersects nontrivially with the dense open subset $F(\moduli)\subseteq \tilde\moduli$, so $F(\moduli^s)$ is non-empty, implying in turn that $\moduli^s$ is non-empty. 
\end{proof} 

Finally we turn to the statement about the local quiver settings of 
$\tilde\moduli$. 

\begin{lemma}\label{lemma:local} For sufficiently large $n$, the local quiver settings of $\moduli$ and $\tilde{\moduli}$ coincide.  
\end{lemma} 

\begin{proof} First we claim that for sufficiently large $n$, the $\theta^{v,n}$-semistable  dimension vectors $\beta\leq\alpha^v$ are exactly the dimension vectors 
$\gamma^v$, where $\gamma$ is a $\theta$-semistable dimension vector with $\gamma\leq\alpha$. Indeed, choose $n$ large enough such that the conclusions of Lemma~\ref{lemma:iotaim} and Lemma~\ref{lemma:imF} hold  
for all $\theta$-semistable dimension vectors $\gamma\leq\alpha$. 
Then $\gamma^v$ is $\theta^{v,n}$-semistable for some $\gamma\leq\alpha$ if and only if $\gamma$ is $\theta$-semistable. So it is sufficient to show that for sufficiently large $n$, if 
$\beta\leq\alpha^v$ is a $\theta^{v,n}$-semistable dimension vector, then 
$\beta(v_-)=\beta(v_+)$. Assume to the contrary that $\beta(v_-)\neq\beta(v_+)$, say $\beta(v_-)>\beta(v_+)$, and choose $n>\sum_{i\in Q_0}\max\{\alpha(i)\theta(i),0\}$. 
Then 
$$\sum_{i\in Q^v_0}\theta^{v,n}(i)\beta(i)=
n(\beta(v_+)-\beta(v_-))+\theta(v)\beta(v_+) +\sum_{i\in Q_0\setminus\{v\}}\theta(i)\beta(i)
\leq -n+\sum_{i\in Q_0}\max\{\alpha(i)\theta(i),0\}<0,$$ 
hence $\beta$ is not $\theta^{v,n}$-semistable. The case 
$\beta(v_-)<\beta(v_+)$ is dealt with similarly. 
So $\gamma\mapsto\gamma^v$ is a one-to-one correspondence between the set of 
$\theta$-semistable dimension vectors $\leq\alpha$ and the set of 
$\theta^{v,n}$-semistable dimension vectors $\leq\alpha^v$. 
Moreover, for a $\theta^{v,n}$-stable dimension vector $\gamma^v$, either there are infinitely many isomorphism classes of $\theta^{v,n}$-stable representations in $\rep(Q^v,\gamma^v)$, 
or there is only one isomorphism class of $\theta^{v,n}$-stable representations 
(since $\moduli(Q^v,\gamma^v,\theta^{v,n})$ is irreducible). 
Since $\moduli^s(Q,\gamma,\theta)$ is a dense open subvariety of $\moduli(Q^v,\gamma^v,\theta^{v,n})$ by the statements we have proved already, 
in the first case there are infinitely many isomorphism classes of $\theta$-stable representations in $\rep(Q,\gamma)$, whereas in the second case there is a single isomorphism class of $\theta$-stable representations in $\rep(Q,\gamma)$. 
Now let $\tau$ be the $\theta^{v,n}$-semistable representation type of some 
$\xi\in\tilde\moduli$. 
Then by the above considerations, $\tau=(\gamma_1^v,m_1;\ldots;\gamma_q^v,m_q)$ for some $\theta$-stable dimension vectors $\gamma_1,\ldots,\gamma_q$ for $Q$. 
Furthermore, there exists a point $\eta\in\moduli$ whose $\theta$-semistable representation type is $\rho:=(\gamma_1,m_1;\ldots;\gamma_q,m_q)$.   
Note finally that the local quiver settings associated to $\xi$ and $\eta$ are the same, since we have the equality  
$\langle \gamma_i,\gamma_j\rangle_Q=
\langle \gamma_i^v,\gamma_j^v\rangle_{Q^v}$ for all $i,j$. 
Conversely, it is straightforward to show that the local quiver setting associated to $\pi(x)\in\moduli(Q,\alpha,\theta)$ is the same as the 
local quiver setting associated to $\tilde\pi(\iota(x))\in\tilde\moduli$. 
\end{proof} 

\begin{corollary} \label{cor:singularities} For sufficiently large $n$, the 
singularities occuring in the 
moduli space $\moduli(Q,\alpha,\theta)$ are the same 
(up to analytic isomorphism)
as the singularities occuring in  $\moduli(Q,\alpha^v,\theta^{v,n})$. 
\end{corollary} 

\begin{proof} There is an \'etale morphism from 
a neighborhood of the image 
${\overline{0}}$ of the zero representation in the algebraic quotient 
$\rep(Q_{\xi},\mu_{\xi})//Gl(\mu_{\xi})$ into 
a neighborhood of $\xi\in\moduli(Q,\alpha,\theta)$  
by Theorem 4.1 in \cite{adriaenssens-lebruyn}  
(in loc. cit. ${\mathrm{char}}(\field)=0$ is assumed and the Luna Slice Theorem \cite{luna} is used; the results extend to positive characteristic by \cite{domokos-zubkov2002}, using \cite{bardsley-richardson}). 
Recall that an \'etale morphism induces isomorphisms of local ring completions. 
Therefore our statement follows from Theorem~\ref{thm:embedding}. 
\end{proof} 

Doubling step-by-step all the vertices in $Q$ one ends up with a bipartite quiver. This construction was used in the literature to reduce the following problems for arbitrary quivers to the case of quivers without oriented cycles: 
computation of the canonical decomposition of dimension vectors in \cite{schofield}, description of generators of the algebra of semi-invariants 
(Theorem 3.2 in \cite{domokos-zubkov2001}), description of 
$\theta$-semistable (stable) dimension vectors (Proposition 6.7 in \cite{derksen-weyman}). Theorem~\ref{thm:embedding} is the moduli space counterpart of these results, accomplishing a proposal attributed to  Le Bruyn on page 374 in \cite{reineke}. 

\begin{example} {\rm (This example shows that although we may double simultaneously all the vertices, we still have to adjust the weight step-by-step going through the vertices, in order to avoid the appearance of singularities of new type.) 
Let $Q$ be the quiver with two vertices $1,2$, and one arrow $a_{ij}$ from $i$ to $j$ for all $(i,j)\in\{1,2\}\times\{1,2\}$. Take the dimension vector  $\alpha:=(1,1)$, and weight $\theta:=(0,0)$. 
Then $\moduli:=\moduli(Q,\alpha,\theta)$ is an affine space of dimension $3$. 
Now double both vertices $1$ and $2$, to get the quiver $\tilde Q$ with four vertices 
$\left(\begin{array}{cc}1_- & 1_+ \\2_- & 2_+\end{array}\right)$, 
an arrow $a_{ij}$ from $i_-$ to $j_+$ for all $(i,j)\in\{1,2\}\times\{1,2\}$, 
and the new arrow $e_i$ from $i_-$ to $i_+$ for $i=1,2$. 
The corresponding dimension vector is $\tilde\alpha:=\left(\begin{array}{cc}1 & 1 \\1 & 1\end{array}\right)$, and consider the weight 
$\tilde\theta:=\left(\begin{array}{cc}-1 & 1 \\-1 & 1\end{array}\right)$. 
Let $y$ denote the point in $\tilde\rep:=\rep(\tilde Q,\tilde\alpha)$ with 
$y(a_{12})=y(a_{21})=1$, $y(a_{11})=y(e_1)=y(a_{22})=y(e_2)=0$. 
Then $y$ is $\tilde\theta$-semistable, and $\tilde\moduli:=\moduli(\tilde Q,\tilde\alpha,\tilde\theta)$ is singular at the point $\xi$ 
corresponding to $y$, as one can see from the local quiver setting of $\xi$ 
(smooth quiver settings were classified in 
\cite{bocklandt}). More explicitly, it is easy to see that 
$\tilde\moduli$ can be identified with the projective variety 
$\{(z_0:z_1:z_2:z_3:z_4)\in{\mathbb{P}}^4\mid z_1z_2-z_3z_4=0\}$ such that 
$\xi$ is identified with the singular point $(1:0:0:0:0)$. 
On the other hand, 
for a point $x\in\rep(Q,\alpha)$ define $\iota(x)\in\tilde\rep$ by $\iota(x)(a_{ij})=x(a_{ij})$ for all $i,j\in\{1,2\}$ and $\iota(x)(e_i)=1$ for $i=1,2$. 
It is easy to see that $\iota$ induces an isomorphism between the affine space $\moduli\cong{\mathbb{A}}^3$ and the dense open subvariety of $\tilde\moduli$ given by $z_4\neq 0$ in the above explicit description of $\tilde\moduli$. 
Finally, we note that replacing the weight $\tilde\theta$ by 
$\sigma:=\left(\begin{array}{cc}-1 & 1 \\-2 & 2\end{array}\right)$ 
one gets a smooth moduli space $\moduli(\tilde Q,\tilde\alpha,\sigma)$. 
}  
\end{example} 

\begin{remark} {\rm (i) As a special case, all the varieties parametrizing semi-simple representations of quivers (cf. \cite{lebruyn-procesi}) can be viewed as open dense subvarieties of projective moduli spaces of bipartite quivers. In particular, the smooth quiver settings classified in \cite{bocklandt} provide examples of smooth projective moduli spaces of representations of quivers. So in a certain sense the quotient spaces of \cite{lebruyn-procesi} 
are brought into the realm of representation spaces of finite dimensional path algebras (i.e. quivers without oriented cycles), despite the fact that the original construction of \cite{lebruyn-procesi}  yields only trivial quotient spaces in the case of finite dimensional path algebras. 

(ii) When $Q$ has no oriented cycles, then $\moduli(Q,\alpha,\theta)$ is a projective variety, hence the morphism induced by $\iota$ in 
Theorem~\ref{thm:embedding} is an isomorphism between 
$\moduli(Q,\alpha,\theta)$ and $\moduli(Q^v,\alpha^v,\theta^{v,n})$ (for sufficiently large $n$). So any moduli space of representations of a quiver  without oriented cycles can be realized as a moduli space of representations of a bipartite quiver. }
\end{remark}


\section{Moduli characterization of tame quivers}
\label{sec:smooth} 

By a {\it connected quiver} we mean a quiver whose underlying graph is connected. The study of representations of a quiver trivially reduces  to the study of representations of the connected components. 

\begin{theorem} \label{thm:main} 
The following are equivalent for a finite connected quiver $Q$: 
\begin{itemize}
\item[(1)] The moduli spaces $\moduli(Q,\alpha,\theta)$ are smooth (possibly  empty)
for all dimension vectors $\alpha$ and weights $\theta$.  
\item[(2)] 
For all $\alpha,\theta$ the moduli space $\moduli(Q,\alpha,\theta)$ is either empty or is a projective space or is an affine space. 
\item[(3)] The underlying graph of $Q$ is Dynkin or extended Dynkin. 
\end{itemize}
\end{theorem} 

\begin{proof} 
The implication $(2)\implies (1)$ is trivial. 

$(1)\implies (3)$: (The argument is a generalization of an example from \cite{adriaenssens-lebruyn}, 
and part of it appears in \cite{reineke}.) 
Recall the Ringel bilinear form on $\mz^{Q_0}$ defined by 
$\langle \alpha,\beta\rangle=\sum_{i\in Q_0}\alpha(i)\beta(i)
-\sum_{a\in Q_1}\alpha(a_-)\beta(a_+)$. 
A dimension vector $\alpha$ is a Schur root (cf. \cite{kac}) if the generic point in $\rep(Q,\alpha)$ corresponds to an indecomposable representation of $Q$. 
Suppose that $Q$ is not Dynkin or extended Dynkin. Then there exists a Schur root $\gamma$ 
with $\langle \gamma,\gamma\rangle <0$ (this follows for example from 
the representation theoretic interpretation of the Ringel form, and Lemma 1.3 and Corollary 2.7 in \cite{kraft-riedtmann}). There exists a weight 
$\theta$ such that there is a $\theta$-stable point in $\rep(Q,\gamma)$ 
(see Theorem 6.1 in \cite{schofield} for an explicit $\theta$, or 
Proposition 4.4 in \cite{king}). Denote by $V$ a representation of $Q$ corresponding to a $\theta$-stable point in $\rep(Q,\alpha)$. 
Then the $3\gamma$-dimensional representation $W:=V\oplus V\oplus V$ is 
$\theta$-semistable. Let $y\in \rep(Q,3\gamma)$ be a point corresponding 
to $W$, so $y\in\rep(Q,3\gamma)^\sstheta$, write $\xi:=\pi(Q,3\gamma,\theta)(y)$. 
By Proposition 4.2 in \cite{adriaenssens-lebruyn}, the point $\xi$ is smooth in 
$\moduli(Q,3\gamma,\theta)$ if and only if the ring of invariants of the local quiver setting of $\xi$ is a polynomial ring (this is proved in loc. cit. under the assumption that ${\mathrm{char}}(\field)=0$ using the Luna Slice Theorem \cite{luna}; 
the results extend to positive characteristic by \cite{domokos-zubkov2002}, using \cite{bardsley-richardson}). 
The local quiver setting of $\xi$ is the one-vertex quiver with 
$1-\langle \gamma,\gamma\rangle\geq 2$ loops and dimension $3$. 
It is well known that the ring of conjugation invariants of $m$-tuples of $3\times 3$ matrices with $m\geq 2$ is not a polynomial ring 
(see \cite{lebruyn-teranishi} for the case ${\mathrm{char}}(\field)=0$, and 
\cite{domokos-kuzmin-zubkov} for positive characteristic). 
Consequently, $\moduli(Q,3\gamma,\theta)$ is singular at its point corresponding to $W$. 

$(3)\implies (2)$: If $Q$ is a Dynkin quiver, then $\rep(Q,\alpha)$ contains a dense open orbit, hence a moduli space $\moduli(Q,\alpha,\theta)$ is either a single point or is empty. 
If $Q$ is extended Dynkin and contains no oriented cycles, then its path algebra is a tame concealed-canonical algebra, and as a special case of a more general result, we get from 
Corollary 7.3 in \cite{domokos-lenzing} that any non-empty moduli space $\moduli(Q,\alpha,\theta)$ is isomorphic to a projective space. If $Q$ is a tame quiver that contains oriented cycles, then the underlying graph of $Q$ is $\widetilde{A}_r$ for some $r\in\mn_0$, 
with the cyclic orientation (i.e. $Q$ has $r+1$ vertices $0,1,\ldots,r$, with an arrow from $0$ to $1$, $1$ to $2$, $2$ to $3$, etc., $r-1$ to $r$, and $r$ to $0$). Take a dimension vector $\alpha$ and weight $\theta$, and apply Theorem~\ref{thm:embedding} with a vertex $v$ with 
$\alpha(v)$ minimal possible. Then the underlying graph of $Q^v$ is $\widetilde{A}_{r+1}$, with a path of length $r+1$ from $v_-$ to $v_+$,  
plus the arrow $e$ from $v_-$ to $v_+$. Choose $n$ as in Theorem~\ref{thm:embedding}. As we pointed out above, $\moduli(Q^v,\alpha^v,\theta^{v,n})$ is a projective space. This shows already the smoothness of $\moduli(Q,\alpha,\theta)$ by Theorem~\ref{thm:embedding}. 
Moreover, the image of the embedding $F$ from the proof of Theorem~\ref{thm:embedding} is the non-zero locus of one of the natural homogeneous coordinates on $\moduli(Q^v,\alpha^v,\theta^{v,n})$ (constructed as the projective spectrum of an algebra spanned by relative invariants) by Theorem 6.1 in \cite{domokos-lenzing}, implying that $\moduli(Q,\alpha,\theta)$ is an affine space. 
\end{proof} 

\begin{remark}\label{remark:dynkinversuseuclidean} 
{\rm Dynkin and extended Dynkin quivers are distunguished by the possible dimensions of their moduli  spaces: whereas any non-empty moduli space of a Dynkin quiver is a single point, an extended Dynkin quiver has a $d$-dimensional moduli space for all non-negative integers $d$ (see for example 
\cite{domokos-lenzing}). Moreover, the above proof shows that a quiver which is neither Dynkin nor extended Dynkin has a singular moduli space of arbitrarily large dimension.  }
\end{remark}


 \end{document}